\documentclass{article}
\usepackage{amsmath,amssymb,amsthm} 

 \numberwithin{equation}{section}
 
 \theoremstyle{plain}
\newtheorem{theorem}{Theorem}[section]
\newtheorem{lemma}[theorem]{Lemma}
\newtheorem{proposition}[theorem]{Proposition}
\newtheorem{corollary}[theorem]{Corollary}

\theoremstyle{definition}

\newtheorem{remark}[theorem]{Remark}
 
\newcommand\eps\varepsilon
\renewcommand\phi\varphi
\newcommand\CE{C\!E}

\newcommand\cwedge{\mathbin{\wedge\hspace{-1.2ex}\wedge}}

\newcommand\C{\mathrel C}
\newcommand\di{{:}} 
\newcommand\tri{{\therefore}}

\newcommand\kk{\mathsf k}
\newcommand\ff{\mathsf f}

\newenvironment{myenum}{\begin{enumerate}

        \setlength{\itemsep}{0pt}}
        {\end{enumerate}}

\newcommand\pritem[1]{\item[\hspace*{4ex}(#1)]}
\newcommand\sbitem[1]{\item[\hspace*{3em}(#1)]}

\begin{document}

    \date{}



\title{Lattice of closure endomorphisms of a Hilbert algebra\thanks{This work was supported by Latvian Science Council, Grant No.\ 27/2012}}
\author{J\=anis C\={\i}rulis  \\ 
   Institute of Mathematics and Computer Science, \\
    University of Latvia, \\
email: \texttt{janis.cirulis@lu.lv}
    }

\maketitle

\begin{abstract}
\noindent
A closure endomorphism of a Hilbert algebra $A$ is a mapping that is simultaneously an endomorphism of and a closure operator on $A$. It is known that the set $\CE$ of all closure endomorphisms of $A$ is a distributive lattice where the meet of two elements is defined pointwise and their join is given by their composition. This lattice is shown in the paper to be isomorphic to the lattice of certain filters of $A$, anti-isomorphic to the lattice of certain closure retracts of $A$, and compactly generated. The set of compact elements of $\CE$ coincides with the adjoint semilattice of $A$; conditions under which two Hilbert algebras have isomorphic adjoint semilattices (equivalently, minimal Brouwerian extensions) are discussed. Several consequences are drawn also for implication algebras. \\

\noindent
\textsf{MSC 2010 Primary:}  03G25 \textsc{secondary:} 06A15, 06D99, 06F35, 08A35  
\end{abstract}

\section{Introduction}  \label{intro}

In the literature, several notions of closure endomorphism of an algebra have been in use. We assume the following definition: an endomorphism of an ordered algebra is its closure endomorphism if it is also a closure operator on the algebra.

Closure endomorphisms of implicative (alias Brouwerian) semilattices were introduced and shown to be useful by Tsinakis in \cite{Tsi} and further studied by the present author in \cite{Tsir}; see also \cite{cle}. It is known, in particular, that the set $CE$ of these endomorphisms is closed under composition $\circ$ and forms even a distributive lattice $(\CE, \circ, \wedge)$ (with meet defined pointwise) embeddable in the filter lattice of $A$ and that closure endomorphisms are precisely the Glivenko operators corresponding to quasi-decompositions of $A$.

Remarkably, closure endomorphisms of implicative semilattices can be described in terms of implication only. This suggests that they could likewise be considered also in Hilbert algebras (known also as positive implication algebras), which are implication subreducts of implicative semilattices. Closure endomorphisms on Hilbert algebras were first mentioned in \cite{multimpl}. More extensively they have been studied by the present author in \cite{mult,adj} and, recently, by Gait\'an in \cite{Gait2}. It has turned out that several general properties of closure endomorphisms of implicative semilattices, as well as of the whole set $\CE$, can be transferred to Hilbert algebras, in part, owing to the view on these algebras as implicative partial semilattices, see \cite{part}. We continue this line on investigation in the present paper.

In the next section we collect some necessary information on Hilbert algebras; however, we assume that the reader is already familiar with the very notion of Hilbert algebra and with elementary arithmetics in these algebras. This information can be found, e.g., in \cite{free,mult,part,Diego,Marsd1,Ras}.  General information about closure endomorphisms on Hilbert algebras, including a few new results, is presented in Section 3. In particular, the closure endomorphisms form a distributive lattice also in this case. In Section 4, this lattice is shown to be isomorphic to the lattice of monomial filters, and anti-isomorphic to the lattice of certain closure retracts of the underlying Hilbert algebra. In the last section attention is paid to the so called finitely generated closure endomorphisms, which form the adjoint semilattice of a Hilbert algebra in the sense of \cite{adj}. The lattice $CE$ (though not necessary complete) turns out to be compactly generated by the finitely generated closure endomorphisms, and the adjoint semilattices of two Hilbert algebras are isomorphic if either the filter lattices or the the endomorphism monoids of these algebras are isomorphic.

Like \cite{mult,part}, the so called dot notation is used in the paper to reduce the number of grouping parentheses in expressions (terms of Hilbert algebras). For instance, any of the expressions
\begin{gather*}
x \to y. \to .z \to x\di \to .(x \to \di y \to .z \to x) \to y, \\
(x \to y. \to .z \to x) \to (x \to \di y \to .z \to x\tri \to y)
\end{gather*}
is a condensed version of
\[
((x \to y) \to (z \to x)) \to ((x \to (y \to (z \to x))) \to y) \enspace .
\]

\section{Hilbert algebras}  \label{prelim}

A \emph{Hilbert algebra} $(A,\to,1)$ may be treated as a poset with the greatest element 1 equipped with a binary operation $\to$ such that
\begin{gather*}
x \to y = 1 \text{ if and only if } x \le 1, \\
x \le y \to x, \quad x \to .y \to z \le x \to y. \to .x \to z \enspace .
\end{gather*}
The join and the meet of elements $a, b \in A$, when they exist, will be denoted by $a \vee b$, resp., $a \wedge b$. A Hilbert algebra is said to be commutative, or an \emph{implication} (or Tarski) \emph{algebra}, if it satisfies any of the equivalent identities
\[
x \to y. \to x = x, \quad x \to y. \to y = y \to x. \to x \enspace .
\]
In an implication algebra, always $x \to y. \to y = x \vee y$, and $x \wedge y$ exists whenever the pair $x,y$ has a lower bound. A \emph{block} of a Hilbert algebra $A$ is any its subalgebra that itself happens to be a bounded implication algebra. By \cite[Theorem 2.1]{part}, a subset $B$ of $A$ is a block if and only if, for some subalgebra $X$ and an element $p \in A$, $B = \{x \to p\colon x \in X\}$.

An \emph{implicative} (or Brouwerian) \emph{semilattice} $(A,\wedge,\to,1)$ is a lower semilattice with the greatest element $1$ in which $a \to b$ is the pseudocomplementation of $a$ relative to $b$:
\[
a \to b \le c \text{ if and only if } a \wedge b \le c ;
\]
see \cite{Cu,N}. An implicative semilattice is always a Hilbert algebra; more exactly, Hilbert algebras are just $(\to,1)$-subreducts of implicative semilattices \cite[Proposition 2.2]{part}.

Let $(A, \to, 1)$ be a Hilbert algebra, and let $\le$ be its natural order relation. The compatibility relation on Hilbert algebras was introduced in \cite{Marsd1}. An equivalent definition is used in \cite{mult,part}: elements $a, b \in A$ are said to be \emph{compatible} (in symbols, $a \C b$) if they have a lower bound $c$ such that $a \le b \to c$. This lower bound is necessary a meet of $a$ and $b$; we call also a meet arising in this way {compatible}. A subset of $A$ is its \emph{relative subsemilattice} if it is closed under existing compatible meets. To emphasize that the meet of $a$ and $b$ is compatible, it will, following \cite{adj}, occasionally be written as $a \cwedge b$. $A$ is an implicative semilattice if and only if all meets in $A$ exist and are compatible (see \cite[Theorem 11]{Marsd2}, also \cite[Section 3]{part}); defined in this way, implicative semilattices have been called also (H)-Hilbert algebras and Hertz algebras. Compatibility of sets containing more than two elements is discussed in \cite{adj}.  See \cite[Proposition 3.1]{part} for the following property of blocks.
\begin{proposition}
Any two elements of a block of an implication algebra $A$ are compatible, and their meet in $A$ coincides with their meet in the block.
\end{proposition}

A \emph{filter} (an implicative filter, a deductive system) of $A$ is a subset $J$ containing $1$ and such that $y \in J$ whenever $x, x \to y \in J$. In particular, the sets $\{1\}$ and $A$ are filters. According to \cite[Lemma 3.2]{part}, $J$ is a filter if and only if it is a semilattice filter, i.e., an upwards closed relative subsemilattice of $A$. An essentially equivalent characteristic of filters is the one given, e.g., in Theorem 3.4 of \cite{HJ}: a filter is a non-empty subset $J$ of $A$ such that $x \le y \to z$ implies $z \in J$ for all $x,y \in J$.

Let $[X)$ stand for the filter generated by a subset $X$, i.e. the least filter including $X$. All filters of a Hilbert algebra form a (complete, hence, bounded) distributive lattice with intersection as meet; we denote by $\sqcup$ the join operation of filters.

Every filter $J$ induces a congruence $\theta_J$ defined by
\[
(a,b) \in \theta_J \text{ if and only if } b \to a \in J \text{ and } a \to b \in J .
\]
We denote by $a/\!J$ the congruence class of $\theta_J$ containing $a$. A filter $J$ is said to be \emph{monomial} if every class $a/\!J$ has the greatest element. For instance, the filters $\{1\}$ and $A$ are monomial.

\begin{remark}
In \cite{Sch0}, Schmidt introduced the notion of monomial congruence for join semilattices. Katri\v{n}\'ak in the dual situation of meet semilattices termed filters corresponding to congruences with greatest element in every congruence class comonomial \cite{comon}; we followed this definition in \cite{cle,Tsir}. However, latter (e.g., in \cite{binpairs}) Schmidt used the term `comonomial' for those equivalences (on arbitrary posets), where every equivalence class had the least element. This meaning of the term now seems to be more common; so, we assume in the present paper Schmidt's definition of a monomial congruence.
\end{remark}

Let $J_a$ stand for the set $\{x\colon x \to a \in J\}$; then $a/J = \{x \in J_a: a \to x \in J\}$. By an ideal in $A$ we mean a hereditary, i.e, downwards closed set (or down-set), which is closed also under existing joins.

\begin{lemma}   \label{cofin}
If $J$ is a filter, then the set $J_a$ is an ideal of $A$. Moreover, if one of the maxima $\max a/\!J$ and $\max J_a$ exists, then the other one exists and both are equal.
\end{lemma}
\begin{proof}
Evidently, $J_a$ is hereditary: if $x \in J_a$ and $y \le x$, then $x \to a \le y \to a \in J$ and $y \in J_a$. If $x,y \in J_a$ and $x \vee y$ exists, then $z := x \to a. \cwedge .y \to a \in J$. Further, $z \le x \to a, y \to a$ and $x \to a. \to a, y \to a. \to a \le z \to a$; so, $x, y \le z \to a$, $x \vee y \le z \to a$  and $z \le z \to a. \to a \le x \vee y \to a$. Thus $x \vee y \to a \in J$ and  $x \vee y \in J_a$. Moreover, $a/\!J$ is a cofinal  subset of $J_a$: if $x \in J_a$ and $x' : = x \to a. \to a$, then $x' \to a \in J$, $a \to x' = 1 \in J$ and, consequently, $x' \in a/\!J$; in addition, $x \le x'$. Therefore, if one of the subsets $a/\!J$ and $J_a$ has the greatest element, then it is the greatest one also in the other. 
\end{proof}

Following \cite{multimpl,mult}, we call a \emph{(right) multiplier} on $A$ any mapping $\phi\colon A \to A$ which satisfies the condition
\begin{equation}    \label{mlt}
\phi(x \to y) = x \to \phi y .
\end{equation}
For example, the identity mapping $\eps\colon x \mapsto x$ and the unit mapping $\iota\colon x \mapsto 1$ are multipliers. Also, for every $p \in A$, the mappings $\alpha_p$, $\beta_p$ and $\delta_p$ defined by
\[
\alpha_px := p \to x, \quad \beta_px := x \to p. \to x, \quad \delta_px := p \to x. \to x
\]
are multipliers. Therefore, multipliers are just operations on $A$ commuting with every translation $\alpha_p$. The following list of  properties of multipliers goes back to \cite[Lemma 3.1]{mult} \newpage

\begin{proposition}    \label{multichar}
For all multipliers $\phi$, $\psi$,
\begin{myenum}
\item $\phi1 = 1$,
\item $x \le \phi x$,
\item $\phi x = \phi x \to x. \to x$,
\item $\phi x =  \phi x \to x. \to \phi x$,
\item $\phi\phi x = \phi x$,
\item $\phi x  = \phi x \to \psi x. \to \phi x$,
\item $\psi\phi x = \phi x \to x. \to \psi x$,
\item $\psi\phi x = \phi\psi x$,
\item $\psi\phi x = \phi x \to \psi x. \to \psi x$.
\end{myenum}
\end{proposition}

Notice that, in virtue of (e), the fixpoint set $F_\phi$ of a multiplier coincides with its range and is a subalgebra of $A$. Its kernel $\{x \in A\colon \phi x = 1\}$ will be denoted by $K\phi$. Evidently, $F_\phi = \{x \in A\colon \phi x \le x\}$ and $K_\phi \cap F_\phi = \{1\}$.

The set $M$ of all multipliers may by ordered pointwise; moreover, it is closed under composition $\circ$ and pointwise defined implication. According to Theorem 3.2 of \cite{mult}, the algebra $(M, \to, \iota)$ is a bounded implication algebra with the least element $\eps$, where
\[
\phi \le \psi \text{ if and only if } \phi \circ \psi = \psi .
\]
For any $x$, the subset $M(x) := \{\phi(x)\colon \phi \in M\}$ is a block; in particular, elements $\phi x$ and $\psi x$ of $A$ with $\phi,\psi \in M$ always are compatible. In effect, $(M,\circ,\wedge,\eps,\iota)$ is a Boolean lattice with pointwise defined meet and complementation defined by $- \phi := \phi \to \eps$ (\cite[Corollary 3.3]{mult}).

\section{Closure endomorphisms} \label{isots}

In what follows, let $A$ be some fixed Hilbert algebra.

The multipliers $\eps, \iota, \alpha_p, \beta_p$ are also examples of closure endomorphisms. On the other hand, every closure operator on $A$ is a multiplier; see Proposition \ref{isotmult1} below. The set $\CE$ of all closure endomorphism is closed under composition and meet; thus, it becomes a sublattice of the lattice of multipliers \cite[Corollary 3.5]{mult}. Therefore, the bounded lattice $(\CE,\circ,\wedge,\eps,\iota)$ is distributive.

Being an endomorphism, every mapping $\phi \in \CE$ satisfies the conditions
\begin{gather}
\mbox{if $x \le y$, then $\phi x \le \phi y$}, \label{iso} \\
\phi(x \to y) = \phi x \to \psi x, \label{end} \\
\mbox{if $x \le y \to z$, then $\phi x \le \phi y \to \phi z$}. \label{leto}
\end{gather}
Like implication-preserving mappings on implicative semilattices, $\phi$ is also meet-preserving (multiplicative) in the following sense:
\begin{equation}    \label{multi}
\mbox{if $x \C y$, then $\phi x \C \phi y$ and $\phi(x \wedge y) = \phi x \wedge \phi y$}.
\end{equation}
Indeed, if  $x \C y$, then, (i) $\phi(x \wedge y)$ is a lower bound of $\phi x$ and $\phi y$; (ii) $x \le y \to .x \wedge y$ and, furthermore,
\(
\phi x \le   \phi y \to \phi(x \wedge y)
\)
by Equation \eqref{leto}. Therefore, $\phi(x \wedge y) = \phi x \cwedge \phi y$. Due to this property, every fixpoint set $F_\phi$ is an example of a relative subsemilattice of $A$.

We list two more conditions on a mapping $\phi$:
\begin{gather}
\phi x \to \phi y = x \to \phi y \label{cls}, \\
\mbox{$\phi x = \alpha_p x$ for an appropriate $p$ (dependent on $x$)} . \label{lc}
\end{gather}
 The following observation is a part of Theorem 4.1 in \cite{mult} (and goes back to \cite[Theorem 3]{multimpl}).

\begin{proposition} \label{isotmult1}
A mapping $\phi\colon A \to A$ is a closure endomorphism if and only if any of the following conditions is fulfilled:
\begin{myenum}
\item $\phi$ is an isotonic multiplier,
\item any two of the identities \textup{\eqref{mlt}, \eqref{end}} and \textup{\eqref{cls}} hold,
\item conditions \textup{\eqref{leto}} and \textup{\eqref{lc}} are satisfied.
\end{myenum}
\end{proposition}

The subsequent characteristics of kernels and fixpoint sets of isotonic multipliers also were announced in \cite{multimpl}. A subset $S$ of $A$ was said to be \emph{special} in \cite{cle,multimpl}, if
\begin{equation}    \label{spec}
\mbox{to every $a \in A$ and $b \in S$, there is $p \in A$ such that $\alpha_p a \in S$ and $\alpha_p b = b$}.
 \end{equation}

\begin{theorem} \label{isotmult2}
The following assertions about a multiplier $\phi$ are equivalent:
\begin{myenum}
\item $\phi$ is isotonic, hence, a closure endomorphism,
\item the kernel $K_\phi$ of $\phi$ is a filter,
\item the fixpoint set $F_\phi$ of $\phi$ is special.
\end{myenum}
\end{theorem}
\begin{proof}
Assume that $\phi$ is a multiplier on $A$. If it is isotonic, then, by the preceding proposition, it is a closure endomorphism; hence, $K_\phi$ is a filter. Moreover, $F_\phi$ is special: if $b \in F_\phi$ and $p: = \phi a \to a$ for some $a$, then, in virtue of Proposition \ref{multichar}(c), $p \to a = \phi a \in F_\phi$, while
\( p \to b = \phi a \to a. \to \phi b = \phi(\phi a \to a. \to b) = \phi a \to \phi a. \to \phi b = \phi b = b \);
see Equations \eqref{mlt} and \eqref{end}.
If, conversely, $K_\phi$ is a filter and $a \le b$, then $x \to a \le x \to b$ for every $x \in A$  and, further, $x \to a \in K_\phi$ only if $x \to b \in K_\phi$. It follows that $x \le \phi a$ only if $x \le \phi b$; thus, $\phi a \le \phi b$, as needed.
At last, if the set $F_\phi$ is special and $a \le b$, then $a \le \phi b$ (Proposition \ref{multichar}(b)) and, for an appropriate $p$, $p \to a \in F_\phi$ and $p \to \phi b = \phi b$. But then $\phi a \le p \to \phi a = \phi(p \to a) = p \to a \le p \to \phi b = \phi b$.
\end{proof}

The kernel of every closure endomorphism $\alpha_p$ is the principal filter $[p)$, and conversely. For this reason, we, following \cite{multimpl,adj}, call these closure endomorphisms \emph{principal}. Let $\CE^\alpha$ stand for the set of all such closure endomorphisms. We say that a subset of $A$ is \emph{$\alpha$-closed} if it is closed under all principal closure endomorphisms. For example, every filter of $A$ is $\alpha$-closed. Observe that an $\alpha$-closed subset of $A$ is always a subalgebra of $A$.

The next characteristic of closure endomorphisms and the corollary to it are suggested by similar observations in \cite{Tsi} for implicative semilattices (Lemma 3.3 and Corollary 3.4 therein). They were extended to Hilbert algebras in Lemma 5 and,  respectively, Proposition 6 of \cite{Gait2}. We provide a more compact proof of the first result.

\begin{lemma}
An endomorphism $\phi$ of $A$ is a closure operator if and only if,  for every idempotent endomorphism $\tau$, the endomorphism $\tau\phi$ also is idempotent.
\end{lemma}
\begin{proof}
Assume that $\phi \in \CE$. If $\tau$ is an idempotent endomorphism, then, using (\ref{mlt}),
\( \tau\phi\tau\phi a \to \tau\phi a =
\tau\phi(\tau\phi a \to a) = \tau(\tau\phi a \to \phi a) = \tau\tau\phi a \to \tau\phi a =  1 \),
from where $(\tau\phi)^2 \le \tau\phi$. Similarly,
\( \tau\phi a \to \tau\phi\tau\phi a = \tau\phi(a \to \tau\phi a) = \tau(a \to \phi\tau\phi a) = \tau a \to \tau\phi\tau\phi a = 1\),
for
\[
a \le \phi a, \ \tau a \le \tau\phi a \le \phi\tau\phi a \ \mbox{ and } \ \tau a = \tau^2a \le \tau\phi\tau\phi a ;
\]
hence, $\tau\phi \le (\tau\phi)^2$. Thus $\tau\phi$ is idempotent. Now assume that the condition of the lemma is fulfilled. Putting $\tau = \eps$, we conclude that the endomorphism $\phi$ is idempotent. With $\tau = \alpha_a$, we then obtain:
\[
a \to \phi a = a \to \phi(a \to \phi a) = a \to (\phi a \to \phi\phi a) = a \to 1 = 1,
\]
from where $a \le \phi a$. As $\phi$ is isotonic, it follows that it is a closure operator.
\end{proof}

\begin{corollary}   \label{endos}
If two Hilbert algebras have isomorphic endomorphism mono\-ids, then they have  isomorphic submonoids of closure endomorphisms.
\end{corollary}

We can say more about closure endomorphisms of implication algebras. The items (b), (c), (e) in the theorem below were proved in \cite{Tsir} for commutative implicative semilattices; see Theorem 11 therein.

\begin{theorem} \label{impla1}
If $A$ is an implication algebra, then
\begin{myenum}
\item every multiplier of $A$ is isotonic, i.e., $\CE = M$,
\item for every $\phi \in \CE$, $\phi(x \vee y) = \phi x \vee \phi y = x \vee \phi y = \phi x \vee y$,
\item join in $CE$ is defined pointwise,
\item for every $\phi \in \CE$, its fixpoint set $F_\phi$ is a filter,
\item $\CE^\alpha$ is a hereditary subset of $\CE$: if $\phi \in \CE$ and $\phi \le \alpha_p$, then $\phi = \alpha_{\phi p \to p}$.
\end{myenum}
\end{theorem}
\begin{proof} We shall apply Proposition \ref{multichar}.

\pritem{a}
Assume that $\phi$ is a multiplier.
If $a \le b$, then $a \le \phi b$, i.e., $a \vee \phi b = \phi b$ and, further, $\phi a \vee \phi b = \phi b \to \phi a. \to \phi a \le \phi b \to a. \to \phi a = \phi(\phi b \to a. \to a) = \phi(a \vee \phi b) = \phi b$. Thus, $\phi$ is isotonic.

\pritem{b}
At first, $\phi(a \vee b) = \phi(a \to b. \to b) = \phi a \to \phi b. \to \phi b = \phi a \vee \phi b$. Also, $\phi a \to \phi b. \to \phi b = a \to \phi b. \to \phi b = a \vee \phi b$, and likewise $\phi(b \vee a) = \phi a \vee b$.

\pritem{c}
By Proposition \ref{multichar}(b), $\psi\phi x = \phi x \vee \psi x$ (in $A$).

\pritem{d}
We already have noticed that, owing to Equation \eqref{multi}, the subset $F_\phi$ is a relative subsemilattice of $A$. It is also upwards closed: if $x \in F_\phi$ and $x \le y$, then $\phi y = \phi(x \vee y) = \phi x \vee y = x \vee y = y$.

\pritem{e}
Suppose that $\phi \in \CE$ and $\phi \le \alpha_p$. Notice that $\phi a \le p \to a$ and, hence, $p \le \phi a \to a$ for every $a$. Then $1 = \phi(\phi a \to a) = \phi(p \vee .\phi a \to a) = \phi p \vee (\phi a \to a) = \phi p \to (\phi a \to a). \to (\phi a \to a) \le \phi p \to p. \to .\phi a \to a = \phi a \to \di\phi p \to p. \to a$, from where $\phi a \le \phi p \to p. \to a$. Conversely, $\phi p \to p. \to a\di \to \phi a = \phi(\phi p \to p. \to a: \to a) = \phi p \to \phi p. \to \phi a\di \to \phi a = 1$ and $\phi p \to p. \to a \le \phi a$. Thus, $\phi a = \phi p \to p. \to a$ (for every $a$), and $\phi = \alpha_{\phi p \to p}$, as required.
\end{proof}

We conclude from (a) that, in an implication algebra, $\CE$ is a Boolean lattice, where the complement $-\phi$ of a closure endomorphism $\phi$ is given by $-\phi x = \phi x \to x$; see the final paragraph of Section \ref{prelim}. This observation allows us to improve (d) observing that  $F_\phi = K_{-\phi}$: for any $a$, $a \in F_\phi$ iff $\phi a \le a$ iff $\phi a \to a = 1$ iff $-\phi a = 1$ iff $a \in K_{-\phi}$. Of course, $\CE$ is also closed under the pointwise defined implication and is even a bounded implication algebra.

In particular, every multiplier $\delta_p$ is a closure endomorphism of any implication algebra $A$, and $\delta_p x = (-\alpha_p) x = p \vee x$. Let $\CE^\delta := \{\delta_p\colon p \in A\}$. The item (e) of the theorem immediately implies that $\CE^\delta$ is an upwards closed subset of $CE$: if $\delta_p \le \phi$, then $-\phi \le \alpha_p$ and $-\phi = \alpha_{(-\phi) p}$, i.e., $\phi = \delta_{(-\phi) p}$.

\section{Two characteristics of the lattice $CE$}

The two theorems of this section are analogues of  Theorems 2 and 3 respectively  in the abstract \cite{cle} stated for implicative semilattices; the results announced there were proved in \cite{Tsir} leaning upon the presence of the total meet operation. A part of the subsequent theorem (its first and third statements) was announced for Hilbert algebras in \cite{multimpl}; see Theorem 4 therein and the discussion following it.

Let again $A$ be an arbitrary Hilbert algebra.

\begin{theorem} \label{kk1}
The transformation  $\kk\colon \phi \mapsto K_\phi$ is an embedding of the bounded lattice $\CE$ into the lattice of filters of $A$. Its range $\kk(CE)$ consists precisely of monomial filters.
If $J = K_\phi$, then, for all $a \in A$, $\phi a = \max\, a/J$.
\end{theorem}
\begin{proof} It consists of several steps. Assume that $\phi,\psi \in \CE$.
\pritem{a}
$K_\phi$ is a filter, for $\phi$ is an endomorphism.
\pritem{b}
$\kk$ is a homomorphism:
    \sbitem{b1}
$K_\eps = \{1\}$ and $K_\iota = A$,
    \sbitem{b2}
$K_{\phi \wedge \psi} = K_\phi \cap K_\psi$, for $\phi a \wedge \psi a = 1$ iff $\phi a = 1 = \psi a$,
    \sbitem{b3}
$K_{\phi \circ \psi} = K_\phi \sqcup K_\psi$: evidently, both $K_\phi$ and $K_\psi$ are  subsets of $K_{\phi \circ \psi}$ by Proposition \ref{multichar}(b) and (\ref{iso}), and then $K_\phi \sqcup K_\psi \subseteq K_{\phi \circ \psi}$. Further, choose any $a \in K_{\phi \circ \psi}$. So, $\psi\phi a = 1$ and, by Proposition \ref{multichar}(g,c), $\phi a \to a. \to \di \psi a \to a. \to a = 1$. As $\phi a \to a \in K_\phi$ and $\psi a \to a \in K_\psi$, these elements both belong to $K_\phi \sqcup K_\psi$, and we conclude by the definition of filter that $a \in K_\phi \sqcup K_\psi$. So $K_{\phi \circ \psi} \subseteq K_\phi \sqcup K_\psi$.

\pritem{c}
$\kk$ is injective: if $K_\phi = K_\psi$, then $\phi a = 1$ iff $\psi a = 1$ for every $a$; by (\ref{mlt}), this implies that, for all $a$, $\psi(\phi a \to a) = 1$ and, further, $\phi a \le \psi a$; likewise $\psi a \le \phi a$. Hence, $\phi = \psi$.

\pritem{d}
$K_\phi$ is monomial: for every $a$, $\phi a$ is the maximal element in $a/K_\phi$. Indeed, $\phi a \in a/K_\phi$, as $\phi(\phi a \to a) = 1 = \phi(a \to \phi a)$. Also, if $x \in a/K_\phi$, then $\phi(x \to a) = 1$ and $\phi a \ge \phi x \ge x$.

\pritem{e}
if $J$ is a monomial filter, then $J \in \kk(CE)$: the mapping $\phi\colon A \to A$ defined by
\begin{equation}    \label{phi/F}
\phi a := \max(a/J) = \max J_a
\end{equation}
(see Lemma \ref{cofin}) is a closure endomorphism and $J = K_\phi$. At first,
\sbitem{e1}
$a \le \phi a$, $\phi\phi a  = \phi a$.  \\
Also, by the definition of $\phi$,
\[
\phi a \to a \in J, \ \text{ and \ if } b \to a \in J, \text{ then } b \le \phi a .
\]
In particular, $a \to \phi b. \to .a \to b = a \to .\phi b \to b \in J$ (for $J$ is $\alpha$-closed), i.e., $a \to \phi b \in J_{a \to b}$, and then $a \to \phi b \le \phi(a \to b)$. From (e1), $\phi a \to \phi b \le a \to \phi b$; therefore,
\sbitem{e2}
$\phi a \to \phi b \le \phi(a \to b)$. \\
To prove the reverse inequality, we first note that
\sbitem{e3} $\phi(a \to b) \to .\phi a \to a \in J$, \\
as $J$ is $\alpha$-closed. Further,
\sbitem{e4} $(\phi(a \to b) \to .\phi a \to a) \to (\phi(a \to b) \to .\phi a \to b) \\
\hspace*{2em} = \phi(a \to b) \to .\phi a \to (a \to b) = \phi a \to (\phi(a \to b) \to (a \to b)) \in J$, \\
since $\phi(a \to b) \to (a \to b) \in J$. As $J$ is a filter, from (e3) and (e4),
\sbitem{e5} $\phi(a \to b) \to .\phi a \to b \in J$.

Now let $u := \phi(a \to b) \to .\phi a \to b\di \to b$. Then (e5) implies that $u \in J_b$, i.e., $u \le \phi b$, from where $\phi \to u \le \phi a \to \phi b$. On the other hand, $\phi a \to u = \phi(a \to b) \to .\phi a \to b\di \to .\phi a \to b$; so, $\phi(a \to b) \le \phi a \to u$. Thus,
\sbitem{e6}
$\phi(a \to b) \le \phi a \to \phi b$. \\
Eventually, (e1), (e2) and (e7) show that $\phi$ is a closure endomorphism, indeed. In addition, $a \in K_\phi$ iff $\phi a = 1$ iff $1 \in J_a$ iff $1 \to a \in J$ iff $a \in J$.

\pritem{f} The last assertion of the theorem follows from (d), (e) and (c).
\end{proof}

By Theorem \ref{isotmult2}, fixpoint sets of closure endomorphisms are special subalgebras of $A$. It was proved in \cite{Tsir} that the special subalgebras of an implicative semilattice form a lattice (which is a sublattice of the lattice of all subalgebras) and that the lattice of closure endomorphisms is dually embedded in it. We have been able to extend this result to Hilbert algebras only with a generalized version of the property of being special and with (Hilbert) subalgebras closed under compatible meets. As a full exposition of these matters would take us aside from the main topic of the paper, we present below a reduced version of the corresponding theorem. It nevertheless improves some results announced without proof in \cite{multimpl}, see Theorem 6 and the subsequent discussion therein.

Notice that, in a Hilbert algebra $A$, every special subset is $\alpha$-closed and is therefore a subalgebra of $A$. Indeed, let $S$ be a special subset of $A$, and choose $b \in S$ and $a \in A$. Then, for some $p$, $\alpha_p(a \to b)$ in $S$ and $\alpha_pb = b$. Hence, $a \to b = a \to \alpha_pb = \alpha_p(a \to b) \in S$.

We shall also need a counterpart of Lemma \ref{cofin}. Let $S$ be any subset of $A$, and let $S^a$ stand for the subset $\{x \to a \in S\colon x \in A\}$. Evidently, $S^a \subseteq [a)_S :=\{x \in S\colon a \le x\}$.

\begin{lemma} \label{coinit}
Suppose that the subset $S$ is special. If one of the minima $\min (a]_S$ and $\min S^a$ exists, then the other exists and both are equal.
\end{lemma}
\begin{proof}
For $S$ special, $S^a$ is a coinitial subset of $[a)_S$: if $y \in [a)_S$, then there is $y' \in S^a$ such that $y' \le y$. Indeed, suppose that $y \in S$ and $a \le y$. By the supposition, there is an element $p$ such that $p \to a \in S$ and $p \to y = y$. Put $y':= p \to a$, then $y'\in S^a$ and $y'\le y$ as $p \to a \le p \to y$.
\end{proof}

Recall that a poset $R$ is a \emph{closure retract} if every its subset $[a)_R$ has the least element. The range of any closure operator is a closure retract, and conversely (see, e.g., \cite{binpairs}).

\begin{theorem} \label{ff1}
The transformation $\ff\colon \phi \mapsto F_\phi$ is an order embedding of the bounded poset $CE$ into the dual of the poset of special subalgebras of $A$. Its range $\ff(CE)$  consists exactly of special closure retracts, which form a lattice with meet $\cap$ and join $\nabla$, where $S\, \nabla \,T = \{x \cwedge y\colon x \in S, y \in T\}$. If $R = F_\phi$, then, for all $a$, $\phi a = \min (a]_R$.
\end{theorem}
\begin{proof} consists of several steps. Assume that $\phi,\psi \in CE$.

\pritem{a} Every subset $F_\phi$ is a closure retract and a subalgebra of $A$. By Theorem \ref{isotmult2}, it is also special.

\pritem{b} $\ff$ is order reversing: if $\phi \le \psi$, then, for $b \in F_\psi$, $\phi b \le \psi b \le b$, so that $b \in F_\phi$; thus, $F_\psi \subseteq F_\phi$. More fully,
    \sbitem{b1} $F_\eps = A$ and $F_\iota = \{1\}$,
    \sbitem{b2} $F_{\phi\circ\psi} = F_\phi \cap F_\psi$, for $\psi\phi x = x$ iff $\psi\phi x \le x$ iff $\phi x \le x$ and $\psi x \le x$ iff $\phi x = x =\psi x$.
    \sbitem{b3} $F_{\phi \wedge \psi} = F_\phi \nabla F_\psi$:  at first, if $z \in F_{\phi \wedge \psi}$, then $z = (\phi \wedge \psi)z = \phi z \wedge \psi z$ with $\phi z \in F_\phi$, $\psi z \in F_\psi$, i.e., $z \in F_\phi \nabla F_\psi$; at second, if $a \in F_\phi \nabla F_\psi$, then $a = a_1 \cwedge a_2$ with $a_1 \in F_\phi$, $a_2 \in F_\psi$ and $a_1, a_2 \in F_{\phi \wedge \psi}$, so that $(\phi \wedge \psi)a = (\phi \wedge \psi)(a_1 \cwedge a_2) =  (\phi \wedge \psi)a_1 \cwedge (\phi \wedge \psi)a_2 = a_1 \cwedge a_2 = a$, i.e., $a \in F_{\phi \wedge \psi}$.

\pritem{c} $\ff$ is injective: a closure operator is uniquely determined by its range. Then $F_\phi \subseteq F_\psi$ (i.e., $F_{\phi \circ \psi} = F_\phi$) implies that $\psi \le \phi$. Together with (b), this shows that $\ff$ dually transfers the order structure of $CE$ to $\ff(CE)$.

\pritem{d} If a closure retract $R$ is special, then $R \in \ff(CE)$: $R$ is the range of the closure operator $\phi$ defined by
\begin{equation}    \label{phi/R}
\phi a:= \min [a)_R = \min R^a
\end{equation}
(see Lemma \ref{coinit}), and it remains only to demonstrate that $\phi$ is an endomorphism. We actually shall prove more---that Proposition \ref{isotmult1}(c) applies. Indeed, suppose that $a \le b \to c$. Being special,  $R$ is $\alpha$-closed, so $a \le b \to \phi c \in R$. Hence, $\phi a \le b \to \phi c$ by Equation \eqref{phi/R}, from where $b \le \phi a \to \phi c \in R$ and again $\phi b \le \phi a \to \phi c$ by the definition of $\phi$. We thus have proved the condition \eqref{leto}. On the other hand, it follows from Equation \eqref{phi/R} that $\phi a = x \to a$ for an appropriate $p$. Therefore,  by Proposition \ref{isotmult1}, $\phi$ is a (closure) endomorphism.

\pritem{e} The last assertion of the theorem follows from (a), (d) and (c).
\end{proof}

\begin{corollary}
The monomial filters of a Hilbert algebra form a sublattice of the lattice of filters, which is order-dual to the lattice of special closure retracts. 
\end{corollary}

It also follows from Theorems \ref{impla1}(a,d) and \ref{kk1} that, in an implication algebra, a monomial filter bounded from below is principal.

\begin{corollary}   \label{ff2}
If $A$ is an implication algebra, then 
\begin{myenum}
\item
the transformation $\ff$ is an embedding of the lattice $CE$ into the dual lattice of filters of $A$,
\item
$\kk(CE)$ and $\ff(CE)$ are mutually dual Boolean lattices, where $F_\phi$ and $K_\phi$ are the complements of each other for every $\phi \in CE$.
\end{myenum}
\end{corollary}
\begin{proof}
Assume that  $A$ be an implication algebra. 
\pritem{a}
By Theorem \ref{impla1}, each $F_\phi$ is a filter, and, owing to item (b) in the proof of Theorem \ref{ff1}, it remains to prove that $F_\phi \nabla F_\psi = F_\phi \sqcup F_\psi$ in the filter lattice of $A$. At first, $F_\phi \nabla F_\psi$ is indeed a filter, for it coincides with the filter $F_{\phi \wedge \psi}$. Next, due to this coincidence, $F_\phi, F_\psi \subseteq F_\phi \nabla F_\psi$. At last, if $F_\phi, F_\psi \subseteq J$ for some filter $J$ and $a \in F_\phi \nabla F_\psi$, then $a = a_1 \cwedge a_2$ for some $a_1 \in F_\phi, a_2 \in F_\psi$, so that $a_1,a_2 \in J$ and $a \in J$ (recall that a filter is a relative subsemilattice of $A$); hence, $F_\phi \nabla F_\psi \subseteq J$. Therefore, $F_\phi \nabla F_\psi$ is indeed the least upper bound of $F_\phi$ and $F_\psi$.
\pritem{b}
This follows from the fact that $CE$ is a Boolean lattice (see Theorem {impla1}(a)), the preceding corollary and the identity  $F_{\phi} = K_{-\phi}$ noticed after the proof of  Theorem {impla1}.
\end{proof}

\section{On finitely generated closure endomorphisms}   \label{finit}

Every closure endomorphism is a join of principal closure endomorphisms; more exactly,
 \begin{equation}   \label{jdense}
 \phi = \bigvee(\alpha_p\colon p \in K_\phi)
 \end{equation}
 (\cite[Proposition 2]{adj}).
A closure endomorphism was said in \cite{adj} to be \emph{finitely generated} if it is a join (i.e., composition) of a finite number of principal closure endomorphisms. Equivalently, $\phi$ is finitely generated if $K_\phi$ is a finitely generated filter. For a finite subset $P = \{p_1,p_2, \ldots, p_n\}$ of $A$, let $\alpha_P$ be the closure endomorphism $\alpha_{p_1} \circ \alpha_{p_2} \circ \cdots \circ \alpha_{p_n}$; for $P = \varnothing$, we set $\alpha_P = \eps$. Evidently, $\alpha_P = \alpha_Q$ iff $P$ and $Q$ generate the same filter; in particular, $\alpha_\varnothing = \alpha_1$.

The set $\CE^f$ of all finitely generated closure endomorphisms is an upper subsemilattice of $\CE$, named the \emph{adjoint semilattice} of $A$ in \cite{adj}. See Proposition 3 of that paper for the following observation, which implies by virtue of Theorem \ref{kk1} that every finitely generated filter is monomial.

\begin{proposition} \label{kk2}
The transformation $\kk$ is an isomorphism of $\CE^f$ onto the upper semilattice of finitely generated filters of $A$.
\end{proposition}

Both semilattices are even subtractive (i.e., they are dual  implicative semilattices); see \cite[Theorem 4]{adj}, resp., \cite[Theorem 2.3]{part}. The subset of principal closure endomorphisms of $A$ is closed under subtraction and form an algebra anti-isomorphic to $A$ (by \cite[Theorem 7]{adj}, the transformation $p \mapsto \alpha_p$ is an order-reversing mapping of $A$ into $CE$, naturally, with the range $\CE^\alpha$, and $\alpha_{p \to q} = \alpha_q - \alpha_p$).
Therefore, any implicative semilattice of which $A$ is a reduct and which itself is anti-isomorphic to $\CE^f$ is a \emph{minimal Brouwerian extension} of $A$; see Section 5 of \cite{adj}.

\begin{remark}
We take the opportunity to correct a sad misprint on p.\ 49 in \cite{adj}: a Brouwerian extension $B$ of a Hilbert algebra $A$ is minimal, i.e., generated by $A$, if and only if every element of $B$ can be presented as a \textsl{meet} of a finite subset of $A$, not a join, as mistakenly said there on line 3 (the subsequent demonstration is correct). Actually, a minimal Brouwerian extension of $A$ is a particular implicative semilattice envelope of $A$ in the sense of \cite{free} and, hence, a free implicative semilattice extension of $A$ \cite[Section 6]{free}. See also \cite{Porta}.
\par
Also, in the proof of Theorem 12, line 4, the attribute `finitely generated' should be inserted before `closure endomorphisms'.
\end{remark}

Recall that an element $a$ of a lattice $L$ is said to be \emph{compact} if every set $X \subseteq L$ such that $a \le \bigvee X$ has a finite subset with the same property. (The lattice in the definition is commonly assumed to be complete; we take the liberty to apply the notion of compactness also to arbitrary lattices, as in \cite{B}). The lattice $L$ is said to be \emph{compactly generated} or algebraic, if every its element is a join of compact elements. In the filter lattice of $A$, the compact elements are precisely the finitely generated filters (see \cite[Theorem 2]{Ch} or, for more details, Section 1 in \cite{glimpse}). As were noticed above, all finitely generated filters of $A$ are monomial; therefore, they remain to be compact in the lattice of monomial filters of $A$.  The converse also holds: by a standard argument, any compact element $J$ of  the lattice of monomial filters is finitely generated. Indeed, $J$ is the join of those principal filters (all monomial) that are included in $J$; by compactness, we can choose a finite number $x_1, x_2, \ldots, x_n$ of elements of $J$ so that $J = [x_1) \sqcup [x_2) \sqcup \cdots \sqcup [x_n) = [x_1,x_2,\ldots,x_n)$. Theorem \ref{kk1}, Proposition \ref{kk2} and the equality (\ref{jdense}) now lead us to the following conclusion.

\begin{corollary}
The lattice $\CE$ is compactly generated, and $\CE^f$ is its set of compact elements.
\end{corollary}

We now move to the question which Hilbert algebras have isomorphic adjoint semilattices.

\begin{theorem} \label{isofilters}
Two Hilbert algebras have isomorphic adjoint semilattices if and only if their filter lattices are isomorphic.
\end{theorem}
\begin{proof}
Suppose that $f$ is an isomorphism between filter lattices of Hilbert algebras $A$ and $B$. Then $f$ yields a bijective connection between the sets of compact elements of both lattices and, consequently, establishes an isomorphism between their semilattices of finitely generated filters, so that the adjoint semilattices of $A$ and $B$ are indeed isomorphic.

The converse follows from Theorem 6 in \cite{adj}, which states that the filter lattice of a Hilbert algebra is isomorphic to the lattice of ideals of its adjoint semilattice.
\end{proof}

\begin{theorem}     \label{adjendos}
Two Hilbert algebras have isomorphic adjoint semilattices if they have isomorphic endomorphism monoids.
\end{theorem}
\begin{proof}
It immediately follows from Corollary \ref{endos} that Hilbert algebras with isomorphic monoids of endomorphisms have also isomorphic lattices of closure endomorphisms and, further, lattices of monomial filters (Theorem \ref{kk1}). Just as in the proof of Theorem \ref{isofilters}, this implies isomorphism of the respective adjoint semilattices.
\end{proof}

If a Hilbert algebra is actually an implicative semilattice (this is the case if and only if all its finitely generated filters are principal), then it is dually isomorphic to its adjoint semilattice (\cite[Corollary 9]{adj}). We thus come to a corollary that essentially improves the main result (Theorem 3.3) of \cite{Tsi} stated, in the terminology of the present paper, for implicative semilattices having only principal closure endomorphisms or, equivalently, only principal monomial filters.

\begin{corollary}
Two implicative semilattices are isomorphic if and only if their endomorphism monoids are isomorphic.
\end{corollary}

Theorem 1 of \cite{Gait1}, which is proved using topological (duality theory) methods, states a similar result for implication algebras. As every implication algebra is an implicative semilattice, the theorem is a particular case of the above result.

We end with some more results on implication algebras.  At first, it immediately follows from (\ref{jdense}) that, for every $\phi \in CE$,
\[
\phi = \bigwedge(\delta_p\colon p \in F_\phi) .
\]
Indeed, $-\phi = \bigwedge(-\alpha_p\colon p \in K_\phi) = \bigwedge(\delta_p\colon p \in F_{-\phi})$; see the paragraph subsequent to Theorem \ref{impla1}.
Further, the discussion subsequent to Proposition \ref{kk2} and  Theorem \ref{impla1}(e) lead us to the conclusion that the transformation $p \mapsto \delta_p$ is an embedding of $A$ into the implication algebra $\CE$, while its range $\CE^\delta$ is an upwards closed subalgebra of $\CE$.

The subsequent characterization of implication algebras is suggested by Theorem 2.3 in \cite{Xu}, stated for BCK-algebras.

\begin{theorem}
Let $A$ be a Hilbert algebra. The following statements are equivalent:
\begin{myenum}
\item $A$ is an implication algebra,
\item the fixpoint set of every $\phi \in CE$ is a filter,
\item the fixpoint set of every $\phi \in CE^f$ is a filter,
\item the fixpoint set of every $\phi \in CE^\alpha$ is a filter.
\end{myenum}
\end{theorem}
\begin{proof}
For (a)$\to$(b), see Theorem \ref{impla1}(d). The transfers (b)$\to$(c)$\to$(d) are trivial. Now assume (d) and choose any $a,b \in A$. At first, $a \le a \to b. \to a$. Further, let $\phi := \alpha_a$ and $c := a \to b. \to a\di \to a$.
Clearly, $a \to b \in F_\phi$ and $a \to b. \to c = 1 \in F_\phi$. By the assumption, then $c \in F_\phi$, i.e., $c = a \to c = 1$. Hence,  $a \to b. \to a \le a$ and, eventually, $a \to b. \to a = a$. As $a$ and $b$ are arbitrary, it follows that $A$ is an implication algebra.
\end{proof}

\begin{theorem} \label{impla3}
If $A$ is an implication algebra, then $\CE^f$ is an ideal of the lattice $\CE$.
\end{theorem}
\begin{proof}
As the set $\CE^f$ is closed under compositions, i.e., joins, it remains only to prove that $\CE^f$ is hereditary in $\CE$. Suppose that $\phi \in \CE$, and assume that $\phi \le \alpha_P$ for some finite $P \subseteq A$. Let $Q := \{\phi p \to p\colon p \in P\}$;  we shall show that $\phi = \alpha_Q$. 

If $P = \varnothing$, then $\phi x \le x$ for all $x$, and $\phi = \alpha_1$. Now consider the general case $P = \{p_1, p_2, \ldots, p_n\}$ with $n > 0$. As the lattice $CE$ is distributive,
\[
\phi = \phi \wedge \alpha_P = \phi \wedge (\alpha_{p_1} \circ \alpha_{p_2} \circ \cdots \circ \alpha_{p_n}) = \phi_1 \circ \phi_2 \circ \cdots \circ \phi_n,
 \]
 where $\phi_i = \phi \wedge \alpha_{p_i}$ for $i = 1,2,\ldots,p_n$. Since $\phi_i \le \alpha_{p_i}$, it follows from Theorem \ref{impla1}(e) that  all closure endomorphisms $\phi_i$ are principal; therefore, $\phi \in CE^f$. More specifically (see the proof of the mentioned theorem),  $\phi_i = \alpha_{\phi_ip_i \to p_i} = \alpha_{\phi p_i \to p_i}$; if $q_i := \phi p_i \to p_i$, then $\phi = \alpha_{q_1} \circ \alpha_{q_2} \circ \cdots \alpha_{q_n} = \alpha_Q$, as needed.
\end{proof}

Therefore, the adjoint semilattice of an implication algebra is in fact a sublattice of $\CE$, generated by the principal closure endomorphisms.

\end{document}